\documentclass[preprint
]{elsarticle}

\usepackage{float}
\usepackage[]{graphicx}
\usepackage{xcolor}
\usepackage{multicol}
\usepackage{amssymb,amsmath,amsthm,ascmac,array,bm}
\usepackage{hyperref}

\journal{Journal of \LaTeX\ Templates}

\newcommand{\argmax}{\mathrm{argmax}}

\newcommand{\al}{\alpha}
\newcommand{\sig}{\sigma}

\newcommand{\gam}{\gamma}
\newcommand{\Gam}{\Gamma}
\newcommand{\mbbh}{\mathbb{H}}

\newcommand{\mbbr}{\mathbb{R}}
\newcommand{\mbbrp}{\mathbb{R}_+}
\newcommand{\mbby}{\mathbb{Y}}

\newcommand{\D}{\Delta}
\newcommand{\del}{\delta}
\newcommand{\ep}{\epsilon}
\newcommand{\mca}{\mathcal{A}}
\newcommand{\mcc}{\mathcal{C}}
\newcommand{\mcf}{\mathcal{F}}
\newcommand{\mci}{\mathcal{I}}
\newcommand{\mcl}{\mathcal{L}}
\newcommand{\sumj}{\sum_{j=1}^{n}}
\newcommand{\E}{\mathbb{E}}
\newcommand{\pr}{\mathbb{P}}
\newcommand{\var}{\mathrm{var}}
\newcommand{\p}{\partial}
\newcommand{\cip}{\stackrel{p}{\to}}
\newcommand{\cil}{\stackrel{\mathcal{L}}{\to}}
\newcommand{\nn}{\nonumber}
\newcommand{\tz}{\theta_0}
\newcommand{\tes}{\hat{\theta}_n}
\newcommand{\aes}{\hat{\al}_n}
\newcommand{\bes}{\hat{\beta}_n}
\newcommand{\ges}{\hat{\gam}_n}
\newcommand{\res}{\hat{\rho}_n}
\newcommand{\tesz}{\hat{\theta}_{0,n}}
\newcommand{\aesz}{\hat{\al}_{0,n}}
\newcommand{\besz}{\hat{\beta}_{0,n}}
\newcommand{\gesz}{\hat{\gam}_{0,n}}
\newcommand{\resz}{\hat{\rho}_{0,n}}

\newcommand{\desz}{\hat{\del}_{0,n}}

\def\ds#1{\displaystyle{#1}}
\def\tcr#1{\textcolor{black}{#1}}

\theoremstyle{default} 
\newtheorem{theorem}{Theorem}[section]

\newtheorem{lemma}[theorem]{Lemma}

\newtheorem{remark}[theorem]{Remark}

\numberwithin{equation}{section}
\allowdisplaybreaks

\begin{document}

\begin{frontmatter}

\title{Estimation of ergodic square-root diffusion under high-frequency sampling}

\author{Yuzhong Cheng}
\address{Graduate School of Mathematics, Kyushu University, Japan}
%
\author{Nicole Hufnagel}
\address{Faculty of Mathematics, TU Dortmund University, Germany}
%
\author{Hiroki Masuda}
\address{Faculty of Mathematics, Kyushu University, 744 Motooka Nishi-ku Fukuoka 819-0395, Japan}
\begin{abstract}
Gaussian quasi-likelihood estimation of the parameter $\theta$ in the square-root diffusion process is studied under high frequency sampling.
Different from the previous study of Overbeck and Ryd\'{e}n(1998) under low-frequency sampling, high-frequency of data provides very simple form of the asymptotic covariance matrix. Through easy-to-compute preliminary contrast functions, a practical two-stage manner without numerical optimization is formulated in order to conduct not only an asymptotically efficient estimation of the drift parameters, but also high-precision estimator of the diffusion parameter. Simulation experiments are given to illustrate the results.
\end{abstract}

\begin{keyword}
CIR process\sep Parameter estimation\sep Gaussian quasi-likelihood \sep High frequency data
\end{keyword}

\end{frontmatter}


\section{Introduction}

\subsection{Objective}

Suppose that $X=X^\theta$ is the square-root diffusion process taking values in $(0,\infty)$:
\begin{equation}
	dX_t = (\al-\beta X_t)dt + \sqrt{\gam X_t}dw_t,
	\label{hm:cir}
\end{equation}
where $X_0>0$ a.s., independent of the standard Wiener process $w$.
The process $X$ is also known as the Cox-Ingersoll-Ross (CIR) model; see, among others, \cite{cir1985} and \cite{GoiYor03}.
We are concerned here with asymptotically efficient estimation of the parameter
\begin{equation}
	\theta:=(\al,\beta,\gam) \in \Theta \subset (0,\infty)^3,
	\nonumber
\end{equation}
when $X$ is observed at $t_j=j h$ where $h=h_n\to 0$ while $T_n:=nh\to \infty$ as $n\to\infty$;
no further condition on the rate of $h\to 0$ is imposed, and the equidistant assumption is just for simplicity.

Historically, parametric estimation of the model \eqref{hm:cir} was studied by:
\begin{itemize}
	\item \cite{Ove98}, \cite{AlaKeb12}, and \cite{AlaKeb13}, when $X$ is \textit{continuously} observed, where they considered asymptotic behaviors of the maximum-likelihood estimation;
	\item \cite{OveRyd97}, when $X$ is observed \textit{at low-frequency}, where the sampling step size $h>0$ is fixed, and they considered not only the moment-matching type estimators, but also the local asymptotic normality.
\end{itemize}
Despite of popularity in applications of the model, however, 
parameter estimation issue has not been fully addressed in case of \textit{high-frequency} sampling, which provides us with quantitative effect of sampling frequency for each parameters, together with simpler form of Fisher-information matrix.
The existing literature, which includes \cite{Kes97} and \cite{Yos11} with the references therein, are mostly concerned with the uniformly ellipticity, which does not hold in \eqref{hm:cir}; even when the diffusion coefficient is not uniformly elliptic, it is quite often assumed that the inverse of the diffusion coefficient can be bounded by a constant multiple of the function $1+|x|^C$ for some $C>0$, hence is not bounded below.

Still, we should mention the previous work \cite{AlaKebTra20}, where the authors deduced the local asymptotic normality (LAN) for $(\al,\beta)$ when the diffusion parameter $\gam$ is assumed to be known; \cite{AlaKebTra20} also considered the cases of $\beta=0$ and $\beta<0$ (the associated statistical experiments are LAQ and LAMN, respectively).
It is well-known \cite{cir1985} that the CIR model has a noncentral chi-squared transition density, hence far from being Gaussian.
Nevertheless, having the LAN result of \cite{AlaKebTra20} in hand, we will see that the Gaussian quasi-likelihood function (GQLF), which is constructed through the small-time Gaussian approximation of the true transition density, is asymptotically efficient especially for estimation of the drift parameters. Even better, the GQLF enables us to effectively bypass numerical optimization.

\medskip

The rest of this paper is organized as follows.
After describing some preliminary facts, in Section \ref{sec_GQLA} we will look at the asymptotic behavior of the Gaussian maximum quasi-likelihood estimator, together with an explicit initial estimator.
Section \ref{nh:sec_simulations} presents some simulation results.

\subsection{Preliminaries}
\label{sec_preliminaries}

Let us describe the basic setup imposed throughout this paper.
Denote by $(\Omega,\mcf,(\mcf_t)_{t\ge 0},\pr)$ the underlying filtered probability space.
It is known that \eqref{hm:cir} admits a unique strong solution, hence for our purpose we may and do suppose that $\mcf_t=\sig(X_0) \vee \sig(w_s:\,s\le t)$.

The zero boundary is non-attracting if $2\al>\gam$, so that $X$ stays positive with probability one, see \cite{cir1985}.
We assume the stronger assumption that the parameter space $\Theta$ is a bounded convex domain, whose compact closure satisfies that
\begin{equation}
	\overline{\Theta}\subset \left\{(\al,\beta,\gam)\in(0,\infty)^3:\, 2\al > 5\gam \right\}.
	\label{hm:param.assumption}
\end{equation}
We will denote by $\tz=(\al_0,\beta_0,\gam_0)\in\Theta$ the true value of $\theta$, and write $\pr_\theta$ for the distribution of $X$ associated with the value $\theta$, with simply writing $\pr$ for $\pr_{\tz}$, which may cause no confusion.
The expectation with respect to $\pr$ will be denoted by $\E$.

Under $\pr_\theta$, $X$ admits the gamma invariant distribution with shape parameter $2\al/\gam$ and scale one $2\beta/\gam$.
We denote the invariant distribution under $\pr_\theta$ by $\pi_\theta(dx)$:
\begin{equation}
	\pi_\theta(dx)=\frac{(2\beta/\gam)^{2\al/\gam}}{\Gam(2\al/\gam)}x^{(2\al/\gam)-1}e^{-(2\beta/\gam)x}I_{(0,\infty)}(x)dx,
	\nonumber
\end{equation}
where $I_A$ denotes the indicator function of a set $A$; we will simply write $\pi_0(dx)=\pi_{\tz}(dx)$.
The $q$th moment of $\pi_\theta$ is given by
\begin{equation}
	\int_0^\infty x^q \pi_\theta(dx)=\frac{\Gam\left(q+(2\al/\gam)\right)}{(2\beta/\gam)^{q}\Gam(2\al/\gam)},
	\nonumber
\end{equation}
which is finite if and only if $q>-2\al/\gam$; in particular,
\begin{align*}
	&\int\left(\frac{1}{\gam_0x}\right)\pi_{0}(dx) = \frac{1}{\gam_0}\frac{2\beta_0}{2\alpha_0-\gam_0},
	\\
	&\int\left(\frac{x}{\gam_0}\right)\pi_{0}(dx) = \frac{1}{\gam_0}\frac{\alpha_0}{\beta_0}.
\end{align*}
Though not essential, we focus on the stationary case throughout this paper, that is, the initial distribution $\mcl(X_0)=\pi_0$ under the true distribution.
Since $X$ is (exponentially) strong-mixing by \cite[Corollary 2.1]{GenJeaLar00}, we then have the ergodic theorem:
\begin{equation}
	\frac{1}{T}\int_0^T f(X_t)dt \cip \tcr{\int f(x) \pi_{0}(dx)}, \qquad T\to\infty,
	\label{hm:conti.LLN}
\end{equation}
for any measurable function $f\in L^1(\pi_0)$, where $\cip$ denotes the convergence in $\pr$-probability.

For later reference, we state the basic tools from \cite[Propositions 3, 4, and 5]{AlaKeb13} and \cite[Section 3]{AlaKebTra20}.

\begin{lemma}
	\label{hm:lem_AKT}
	\begin{enumerate}
		\item For each $p<\frac{2\alpha}{\gamma}$,
		\begin{equation}
			\sup_{t\in\mbbrp}\E\left(X_t^{-p}\right) < \infty.
			\nonumber
		\end{equation}
		
		\item For each $p\ge 1$, there exists a constant $C_p>0$ such that
		\begin{equation}
			\sup_{s,t\in\mbbrp:\,0<|t-s|<1}\E\left(|X_t-X_s|^{p}\right) \le C_p|t-s|^{p/2}.
			\nonumber
		\end{equation}
		
		\item For each $p<\frac12\left(\frac{2\alpha}{\gamma}-1\right)$,
		\begin{equation}
			\frac1n \sumj X_{t_{j-1}}^{-p} \cip \int_{0}^{\infty}x^{-p}\pi_{0}(dx)
			=\frac{(2\beta/\gam)^{p} \Gam((2\al/\gam) -p)}{\Gam(2\al/\gam)}.
			\label{hm:LLN}
		\end{equation}
		
	\end{enumerate}
\end{lemma}

By Lemma \ref{hm:lem_AKT} with \eqref{hm:conti.LLN} it is routine to deduce that
\begin{equation}
	\frac1n \sumj f(X_{t_{j-1}}) \cip \int_{0}^{\infty}f(x)\pi_{0}(dx)
	\nn
\end{equation}
for each $\mcc^1$-function with $f(x)$ and $\p_x f(x)$ being of at most polynomial growth for $|x|\to\infty$.
We also have $(2\al/\gam -1)/2 > 2$ under \eqref{hm:param.assumption}, hence in particular \eqref{hm:LLN} holds true for $p=1,2$.

Applying the integration by parts we have
\begin{equation}
	X_t = e^{-\beta(t-s)}X_s + \frac{\al}{\beta}(1-e^{-\beta(t-s)}) 
	+ \sqrt{\gam} e^{-\beta t} \tcr{\int_{s}^{t}e^{u\beta}\sqrt{X_u}dw_u}, \quad t>s.
	\label{hm:X_sol}
\end{equation}
Let the symbol $\E_\theta^{j-1}(\cdot)$ stand for the conditional expectation under $\pr_{\theta}$ with respect to $\mcf_{t_{j-1}}$.
We have $\E_\theta^{j-1}(\int_{t_{j-1}}^{t_{j}} e^{\beta s} \sqrt{X_s} dw_s)=0$ a.s. since
\begin{align}
	\E_\theta^{j-1}\left(\left\langle\int_{t_{j-1}}^{\cdot} e^{\beta s} \sqrt{X_s} dw_s\right\rangle_{t_{j}}\right)
	&=\E_\theta^{j-1}\left(\int_{t_{j-1}}^{t_{j}} e^{2\beta s} X_s ds\right) \nn\\
	&=\int_{t_{j-1}}^{t_{j}} e^{2\beta s} \E_\theta^{j-1}(X_s) ds
	< \infty\quad \text{a.s.}
	\nonumber
\end{align}
by the stationarity.
This together with \eqref{hm:X_sol} leads to the explicit expressions of the conditional mean and the conditional variance:
\begin{align}
	\mu_{j-1}(\al,\beta)&:=\E^{j-1}_\theta(X_{t_{j}}) \nn\\
	&= e^{-\beta h}X_{t_{j-1}} + \frac{\al}{\beta}(1-e^{-\beta h}),
	\label{hm:cond.mean}
	\\
	\sig^2_{j-1}(\theta) &:=\var^{j-1}_\theta(X_{t_{j}}) \nn\\
	&= \gam e^{-2\beta t_j}\E_\theta^{j-1}\left( \int_{t_{j-1}}^{t_j}e^{2s\beta} X_s ds \right)
	\nn\\
	&= \gam e^{-2\beta t_j} \int_{t_{j-1}}^{t_j}\left\{ e^{2s\beta}\left(
	e^{-\beta(s-t_{j-1})}X_{t_{j-1}} + \frac{\al}{\beta}(1-e^{-\beta(s-t_{j-1})})\right)\right\} ds
	\nn\\
	&= \frac{\gam}{\beta}(1-e^{-\beta h}) \left( 
	e^{-\beta h} X_{t_{j-1}} +\frac{\al}{2\beta}(1-e^{-\beta h}) \right).
	\label{hm:cond.var}
\end{align}
Note that the conditional mean is free from the diffusion parameter $\gam$.

%

\section{Gaussian quasi-likelihood}
\label{sec_GQLA}

We denote by $\mbbh_n(\theta)$ the GQLF, which is defined through approximating the conditional distribution $\mcl(X_{t_j}|X_{t_{j-1}})$ under $\pr_\theta$ by $N(\mu_{j-1}(\theta), \sig^2_{j-1}(\theta))$ with the expressions \eqref{hm:cond.mean} and \eqref{hm:cond.var}:
\begin{align}
	\mbbh_n(\theta) &:= \sumj \log \phi\left( X_{t_j};\, \mu_{j-1}(\al,\beta),\, \sig^2_{j-1}(\theta)\right), \nonumber\\
	&= C_n -\frac12 \sumj \left(
	\log\sig^2_{j-1}(\theta) + \frac{1}{\sig^2_{j-1}(\theta)}(X_{t_j} - \mu_{j-1}(\al,\beta))^2 \right) \nn\\
	&=C_n-\frac12 \sumj \log\left\{\frac{\gam}{\beta}(1-e^{-\beta h}) \left( 
	e^{-\beta h} X_{t_{j-1}} +\frac{\al}{2\beta}(1-e^{-\beta h}) \right)\right\}
	\nn\\
	&\qquad{}-\frac12 \sumj \frac{\left(X_{t_j} - e^{-\beta h}X_{t_{j-1}} - \frac{\al}{\beta}(1-e^{-\beta h})\right)^2}{\frac{\gam}{\beta}(1-e^{-\beta h}) \left( 
		e^{-\beta h} X_{t_{j-1}} +\frac{\al}{2\beta}(1-e^{-\beta h}) \right)},
	\label{hm:def_GQLF}
\end{align}
where $C_n$ is a constant which does not depend on $\theta$, hence irrelevant to optimization, 
and where $\phi(\cdot;\mu,\sig^2)$ denotes the Gaussian density with mean $\mu$ and variance $\sig^2$.
Then, we define the Gaussian quasi-maximum likelihood estimator (GQMLE) by any
\begin{equation}
	\tes=(\aes,\bes,\ges) \in \mathop{\mathrm{argmax}}_{\overline{\Theta}}\mbbh_n.
	\label{hm:def_GQMLE}
\end{equation}
It is well-known that the GQLF effectively works for uniformly elliptic diffusions with coefficients smooth enough (see \cite{Kes97}, \cite{Yos11}). As for the present model \eqref{hm:cir}, we need to take care of the irregularity of the diffusion coefficient when proving moment estimates and basic limit theorems.

The GQMLE $\tes$ cannot be given in a closed form, because of its nonlinearity in the parameters.
Nevertheless, as in \cite[Sect 3]{OveRyd97} we can proceed with initial estimator which we will denote by $\tesz=(\aesz,\besz,\gesz)$. The asymptotics is rather similar \cite{OveRyd97}, while we need to take care about some moment estimates in the present high-frequency setup.

Let
\begin{align}
	D_n := \mathrm{diag}\big(\sqrt{T_n},\,\sqrt{T_n},\,\sqrt{n}\big),
	\label{hm:def_D}
\end{align}
and
\begin{align}
	\mci(\theta) &:=
	\begin{pmatrix}
		\int\left(\frac{1}{\gam x}\right)\pi_{\theta}(dx)&-\frac{1}{\gam}&0\\
		-\frac{1}{\gam}&\int\left(\frac{x}{\gam}\right)\pi_{\theta}(dx)&0\\
		0&0&\frac{1}{2}
		\int\left(\frac{1}{\gam^2}\right)\pi_{\theta}(dx)
	\end{pmatrix} 
	\nn\\
	&=\begin{pmatrix}
		\frac{1}{\gam}\frac{2\beta}{2\alpha-\gam}&-\frac{1}{\gam}&0 \\
		-\frac{1}{\gam}&\frac{1}{\gam}\frac{\alpha}{\beta}&0 \\
		0&0&\frac{1}{2\gam^2}
	\end{pmatrix}.
	\label{hm:def_I}
\end{align}
and note that $\mci(\theta)$ is invertible for $\theta\in\overline{\Theta}$:
\begin{align*}
	\mathcal{I}(\theta)^{-1} =
	\begin{pmatrix}
		\frac{\alpha(2\alpha-\gam)}{\beta}&2\alpha-\gam&0\\
		2\alpha-\gam&2\beta&0\\
		0&0&2\gam^2\\
	\end{pmatrix};
\end{align*}
formally, $D_n$ and $\mci(\theta)$ respectively correspond to the optimal rate of convergence for the regular estimators and the Fisher information matrix in case of uniformly elliptic diffusions \cite{Gob02}.

Let $\p_a:=\p/\p a$, the partial-differentiation operator with respect to a variable $a$, with the $k$-fold operation being denoted by $\p_a^k$.
We will first consider a $D_n$-consistent initial estimator $\tesz=(\aesz,\besz,\gesz)$ given in a closed form as in \cite{OveRyd97}, and then construct a Newton-Raphson and/or a scoring one-step estimator toward the GQMLE:
\begin{itemize}
	\item Newton-Raphson method
	\begin{equation}
		\hat{\theta}_{n}^{(1,1)}=\tesz-\left(\p_{\theta}^2\mbbh_n(\tesz)\right)^{-1}\p_{\theta}\mbbh_n(\tesz)
		\nn
	\end{equation}
	\item Method of scoring
	\begin{equation}
		\hat{\theta}_{n}^{(1,2)}=\tesz+D_{n}^{-1}\mathcal{I}(\tesz)^{-1}D_{n}^{-1}\p_{\theta}\mbbh_n(\tesz)
		\nn
	\end{equation}
\end{itemize}
It will turn out that the GMQLE $\tes$ and the one-step estimators $\hat{\theta}_{n}^{(1,1)}$ and $\hat{\theta}_{n}^{(1,2)}$ are all asymptotically equivalent at rate $D_n$.

\subsection{Explicit initial estimator}
\label{hm:sec_preliminary.est}

First, we look at the drift parameter $(\al,\beta)$.
We introduce the conditional least-squares estimator $(\aesz,\besz)$ defined to be a maximizer of
\begin{align*}
	\mbbh_{1,n}(\al,\beta) := -\sumj \left(X_{t_{j}}-\mu_{j-1}(\alpha,\beta)\right)^2,
\end{align*}
given in the closed forms
\begin{align}
	\aesz &= \frac{\bar{X}_n-e^{-\besz h}\bar{X}_n'}{1-e^{-\besz h}}\besz,
	\label{hm:aesz_def}
	\\
	\besz &= -\frac1h \log \left(\frac{\sumj (X_{t_{j-1}}-\bar{X}_n')(X_{t_{j}}-\bar{X}_n)}{\sumj (X_{t_{j-1}}-\bar{X}_n')^2}\right),
	\label{hm:besz_def}
\end{align} 
where $\bar{X}_n := n^{-1} \sumj X_{t_j}$ and $\bar{X}_n' := n^{-1} \sumj X_{t_{j-1}}$.

As for the diffusion parameter $\gam$, we substitute $(\aesz,\besz)$ into $(\al,\beta)$ in the GQLF \eqref{hm:def_GQLF} and denote the resulting function by $\mbbh_{2,n}(\gam)$:
\begin{align*}
	\mbbh_{2,n}(\gam) &:= \mbbh_n\big(\aesz,\besz,\gam\big) \nn\\
	&=\sumj \log \phi\left( X_{t_j};\, \mu_{j-1}(\aesz,\besz),\, \sig^2_{j-1}(\aesz,\besz,\gam)\right) \nn\\
	&= C_n -\frac12 \sumj \bigg(
	\log\sig^2_{j-1}(\aesz,\besz,\gam) \nn\\
	&{}\qquad	+ \frac{1}{\sig^2_{j-1}(\aesz,\besz,\gam)}(X_{t_j} - \mu_{j-1}(\aesz,\besz))^2 \bigg).
\end{align*}
Then, we define $\gesz$ to be a maximizer of $\mbbh_{2,n}$:
\begin{align}
	\label{hm:gesz_def}
	\gesz = \frac1n \sumj \frac{(X_{t_{j}}-\mu_{j-1}(\aesz,\besz))^2}{\frac{1}{\besz}(1-e^{-\besz h}) \left( 
		e^{-\besz h} X_{t_{j-1}} +\frac{\aesz}{2\besz}(1-e^{-\besz h})\right)}.
\end{align}

\begin{lemma}
	\label{hm:lem_cLSE}
	$D_n(\tesz-\tz) = O_p(1)$.
\end{lemma}

\begin{proof}
	Let $\rho:=(\al,\beta)\in\Theta_\rho\subset(0,\infty)^2$, with $\Theta_\rho$ denoting the parameter space of $\rho$, and let $\resz:=(\aesz,\besz)$.
	Let $B_n(\rho):=(e^{-\beta h}, \frac{\al}{\beta}(1-e^{-\beta h}))^\top$, $x_{j-1}:=(X_{t_{j-1}},1)^\top$, $\bm{x}_n:=(x_0,x_1,\dots,x_{n-1})^\top$, and $\bm{y}_n:=(X_{t_1},\dots,X_{t_n})$.
	Then $\mbbh_{1,n}(\rho)=-\| \bm{y}_n - \bm{x}_n B_n(\rho)\|^2$ and the corresponding estimating equation $\p_\rho\mbbh_{1,n}(\rho)=0$ is equivalent to $\p_{B_n}\mbbh_{1,n}(\rho)=0$.
	The solution $\hat{B}_n = B_n(\res)$ is given by $B_n(\res) = (\bm{x}_n^\top \bm{x}_n)^{-1} \bm{x}_n^\top \bm{y}_n$.
	Note that
	\begin{equation}
		\frac1n \bm{x}_n^\top \bm{x}_n \cip \int
		\begin{pmatrix}
			x^2 & x \\ x & 1
		\end{pmatrix}
		\pi_0(dx) >0.
		\nonumber
	\end{equation}
	By \eqref{hm:X_sol}, we have $\bm{y}_n=\bm{x}_n B_n(\rho_0) + M_n$ under $\pr$, where $M_n =(M_{n,j})_{j=1}^{n}$ with
	\begin{equation}
		M_{n,j} := \gam_0 \int_{t_{j-1}}^{t_j} e^{-(t_j - s)\beta_0} \sqrt{X_s} dw_s.
		\nonumber
	\end{equation}
	Since $\{(M_{n,j},\mcf_{t_{j}})\}_{j\le n}$ forms a martingale-difference array satisfying that
	\begin{equation}
\sup_{n,j}\E\left(\left|h^{-1/2}M_{n,j}\right|^q\right)<\infty
\nonumber
\end{equation}
	for every $q>0$, we have
	\begin{equation}
		\sqrt{\frac{n}{h}} \left( B_n(\res) - B_n(\rho_0) \right)= \left( \frac1n\bm{x}_n^\top \bm{x}_n \right)^{-1}
		\frac{1}{\sqrt{n}} \sumj \frac{1}{\sqrt{h}}M_{n,j} x_{j-1} = O_p(1).
		\label{hm:lem_cLSE-p1}
	\end{equation}
	Apply the delta method to \eqref{hm:lem_cLSE-p1} with the mapping $(x,y) \mapsto (\log x,y)$ to obtain
	\begin{equation}
		\sqrt{\frac{n}{h}} \left(
		(-\besz h) - (-\beta_0 h),~\frac{\aesz h }{\besz h}(1-e^{-\besz h}) - \frac{\al_0 h }{\beta_0 h}(1-e^{-\beta_0 h})
		\right) = O_p(1).
		\nn
	\end{equation}
	Another application of the delta method to the last display with the mapping $(x,y) \mapsto (xy/(e^x -1), -x)$ yields
	$\sqrt{n/h}( \aesz h -\al_0 h,~\besz h -\beta_0 h) = O_p(1)$, followed by $\sqrt{T_n} (\resz -\rho_0) = O_p(1)$.
	
	\medskip
	
	To deduce that $\sqrt{n}(\gesz-\gam_0)=O_p(1)$, we recall the expressions \eqref{hm:cond.var} and \eqref{hm:gesz_def}.
	\tcr{For convenience we introduce the following notation:
		\begin{align}
		c(x,\rho)&=\gam^{-1}\sig^2(x,\theta), \nn\\
		\pi'(x,\rho)&=2\sqrt{h} c(x,\rho)^{-1}\{\mu(x,\rho_0)-\mu(x,\rho)\}, \nn\\
		\pi''(x,\rho)&=h c(x,\rho)^{-1}, \nn\\
			\chi'_{j}(\rho) &:= h^{-1/2}\left(X_{t_{j}} -\mu_{j-1}(\rho)\right), 
			\label{hm:def_chi1}
			\\
			\chi''_{j}(\theta) &:= \frac1h \left\{\left(X_{t_{j}} -\mu_{j-1}(\rho)\right)^2 - \gam c_{j-1}(\rho)\right\}.
			\nn
		\end{align}
		Then, for any $q>0$, the family of random variables $(\chi'_{j})_{j\le n}=(\chi'_{j}(\rho_0))_{j\le n}$ and $(\chi''_{j})_{j\le n}=(\chi''_{j}(\theta_0))_{j\le n}$ are two $L^q(\pr_\theta)$-martingale-difference arrays with respect to the filtration $(\mcf_{t_j})$.} 
	We can write
	\begin{align}
		\sqrt{n}(\gesz-\gam_0) &= \frac{1}{\sqrt{n}}\sumj \left\{\frac{\left(X_{t_{j}}-\mu_{j-1}(\resz)\right)^2}{c_{j-1}(\resz)}-\gamma_0\right\}
		\nn\\
		&= G_{1,n}(\resz)+G_{2,n}(\resz)+G_{3,n}+G_{4,n},
		\nn
	\end{align}
	where
	\begin{align}
		& G_{1,n}(\rho) := \frac{1}{\sqrt{n}}\sumj \pi'_{j-1}(\rho)\, \chi'_{j}, \qquad 
		G_{2,n}(\rho) := \frac{1}{\sqrt{n}}\sumj \pi''_{j-1}(\rho)\, \chi''_{j},
		\nn\\
		& G_{3,n} := \frac1n \sumj \pi''_{j-1}(\resz) \frac{\sqrt{n}}{h}\left(\mu_{j-1}(\rho_0) - \mu_{j-1}(\resz)\right)^2
		\nn\\
		and
		\nn\\
		& G_{4,n} := \frac1n \sumj \pi''_{j-1}(\resz) \frac{\sqrt{n}}{h}\gamma_0\left(c_{j-1}(\rho_0) - c_{j-1}(\resz)\right).
		\nonumber
	\end{align}
	In what follows, we will write $a_n \lesssim b_n$ if $\sup_n(a_n/b_n)\le c$ for some universal constant $c>0$; even when $a_n$ and $b_n$ are non-negative random variables, we will use the same symbol $a_n \lesssim b_n$ if the inequality holds a.s.
	
	To estimate $G_{3,n}$ and $G_{4,n}$, we recall \eqref{hm:cond.mean} and note that 
	$\sup_\rho |\pi''_{j-1}(\rho)| \lesssim X_{t_{j-1}}^{-1}$. We have $G_{3,n}=O_p(n^{-1/2})$, since
	\begin{align}
		|G_{3,n}| &\lesssim \frac{1}{\sqrt{n}} \Bigg\{ \bigg(\frac1n \sumj X_{t_{j-1}}^{-1} \bigg) \left|\sqrt{T_n}(\aesz -\al)\right|^2 \nn\\
		&{}\qquad 
		+\bigg(\frac1n \sumj (X_{t_{j-1}} + X_{t_{j-1}}^{-1}) \bigg) \left|\sqrt{T_n}(\besz -\beta_0)\right|^2 \Bigg\}
		\nonumber
	\end{align}
	and the terms inside the curly bracket are $O_p(1)$ by the law of large numbers \eqref{hm:LLN}.
	Likewise, by \eqref{hm:cond.var},
	\begin{align}
		|G_{4,n}| \lesssim \frac1n \sumj X_{t_{j-1}}^{-1}  \left|\sqrt{T_n}(\resz -\rho_0)\right| \sqrt{h},
		\nonumber
	\end{align}
	so that $G_{4,n}=O_p(n^{-1/2})$. 
	
	It remains to deduce that both $G_{1,n}(\resz)$ and $G_{2,n}(\resz)$ are $O_p(1)$.
	Regard $G_{1,n}(\rho)$ and $G_{2,n}(\rho)$ as stochastic processes in $\mcc(\overline{\Theta_\rho})$.
	Since Burkholder's inequality ensures that $G_{i,n}(\rho)=O_p(1)$ for each $i=1,2$ and $\rho\in\overline{\Theta_\rho}$, it suffices to verify the tightness of $G_{i,n}(\cdot)$ in $\mcc(\overline{\Theta_\rho})$.
	In view of the Kolmogorov tightness criterion (e.g. \cite{IbrHas81}), it is in turn sufficient to show the moment estimate
	\begin{equation}
		\exists \del,K,C>0~\forall\rho,\rho',\quad 
		\sup_n \E\left(\left|G_{1,n}(\rho)-G_{1,n}(\rho')\right|^K\right) \le C |\rho-\rho'|^{2+\del}.
		\label{hm:lem_cLSE-p2}
	\end{equation}
	
	Elementary calculations give
	\begin{align}
		\sup_\rho |\p_\rho \pi'_{j-1}(\rho)| &\lesssim \sqrt{h}\big(1+X_{t_{j-1}}^{-1}+X_{t_{j-1}}^{-2}\big).
		\nonumber
	\end{align}
	By the standing assumption that $2\al > 5\gam$ for each $\rho\in\overline{\Theta_\rho}$ and Lemma \ref{hm:lem_AKT}, we can find a sufficiently small $\del>0$ such that, for $K=2+\del$,
	\begin{equation}
		\sup_{j\le n}\E\left((1+X_{t_{j-1}})^{C'}\sup_\rho |\p_\rho \pi'_{j-1}(\rho)|^K \right) = O(1).
		\nonumber
	\end{equation}
	By means of Burkholder's and Jensen's inequalities and using the estimate that $\E_{\tz}^{j-1}(|\chi'_j|^K) \lesssim (1+X_{t_{j-1}})^{C'}$ a.s. for some universal constant $C'=C'(K)$, we see that
	\begin{align}
		& \sup_n \E\left(\left|G_{1,n}(\rho)-G_{1,n}(\rho')\right|^K\right) \nn\\
		&= \sup_n \E\left(\left|\frac{1}{\sqrt{n}} \sumj \left(\pi'_{j-1}(\rho)-\pi'_{j-1}(\rho')\right) \chi'_j \right|^K\right) \nn\\
		&\lesssim \sup_n \E\left(\left| \frac1n \sumj \left(\pi'_{j-1}(\rho)-\pi'_{j-1}(\rho')\right)^2 (\chi'_j)^2 \right|^{K/2}\right) \nn\\
		&\lesssim \sup_n \frac1n \sumj \E\left\{\left(\pi'_{j-1}(\rho)-\pi'_{j-1}(\rho')\right)^K \E_{\tz}^{j-1}(|\chi'_j|^K) \right\} \nn\\
		&\lesssim \sup_n \frac1n \sumj \E\left((1+X_{t_{j-1}})^{C'}\sup_\rho |\p_\rho \pi'_{j-1}(\rho)|^K \right) |\rho-\rho'|^{K} \nn\\
		&\lesssim |\rho-\rho'|^{K}.
		\nonumber
	\end{align}
	This verifies \eqref{hm:lem_cLSE-p2} for $G_{1,n}$. In a similar manner we can deduce \eqref{hm:lem_cLSE-p2} for $G_{2,n}$ by using
	\begin{align}
		\sup_\rho |\p_\rho \pi''_{j-1}(\rho)| &\lesssim
		h \big(X_{t_{j-1}}^{-1} + X_{t_{j-1}}^{-2}\big).
		\nonumber
	\end{align}
	Thus, we have seen that both $G_{1,n}(\cdot)$ and $G_{2,n}(\cdot)$ are tight in $\mcc(\overline{\Theta_\rho})$.
	Hence $\sqrt{n}(\gesz-\gam_0)=O_p(1)$, completing the proof.
\end{proof}

We could prove the asymptotic normality of $D_n(\tesz-\tz)$ directly by means of the central limit theorem for martingale difference arrays, though it does not play an important role in our study.


\subsection{Asymptotics for joint GQMLE}

Recall the definitions \eqref{hm:def_GQLF}, \eqref{hm:def_GQMLE}, and \eqref{hm:def_D}.
The objective of this section is to deduce the asymptotic normality of the GQMLE.
Let $\cil$ denote the convergence in law, and recall \eqref{hm:def_I} for the definition of $\mci(\tz)$.

\begin{lemma}
	\label{hm:lem_AN.GQMLE}
	We have
	\begin{equation}
		D_n(\tes-\tz) \cil N_p\left(0,\, \mci(\tz)^{-1}\right).
		\label{hm:gqmle-t2}
	\end{equation}
\end{lemma}

Trivially we have
\begin{equation}
	\mci(\tes)^{1/2}D_n(\tes-\tz) \cil N_p\left(0,\, I_3\right),
	\nonumber
\end{equation}
where $I_p$ denotes the $p$-dimensional identity matrix.
The scenario of the proof is much the same as in the classic uniformly-elliptic diffusion models as in \cite{Kes97}, except that we need to take care about tightness/integrability issues caused by the diffusion-coefficient form.

Before proceeding to the proof, we introduce some notation and make a few remarks.
Given a random function $\zeta_n(\theta)$ and a real sequence $r_n>0$, we will write $\zeta_n(\theta)=o_p^\ast(r_n)$ and $\zeta_n(\theta)=O_p^\ast(r_n)$ if $\sup_\theta |r_n^{-1}\zeta_n(\theta)|\cip 0$ and $\sup_\theta |r_n^{-1}\zeta_n(\theta)|=O_p(1)$, respectively; we will analogously write $\zeta_n(\theta)=o^\ast(r_n)$ and $\zeta_n(\theta)=O^\ast(r_n)$ when $\zeta_n$ is non-random.
We also note that under \eqref{hm:LLN} the following uniform law of large numbers is in force:
\begin{equation}
	\frac1n \sumj f_{j-1}(\rho) \cip \int f(x,\rho)\pi_0(dx)
	\label{hm:gqmle-t3}
\end{equation}
uniformly in $\rho=(\al,\beta)$ for any measurable $f$ such that
\begin{equation}
	\max_{k=0,1}\sup_\rho |\p_{\rho}^k f(x,\rho)| \lesssim |x|^{-2} \vee (1+|x|^K), \qquad x>0,
	\nonumber
\end{equation}
for some $K\ge 0$; this does hold, since we have the tightness 
\begin{equation}
\sup_\rho \left| \frac1n\sumj \p_\rho f_{j-1}(\rho)\right| = O_p(1)
\nonumber
\end{equation}
in addition to the $\rho$-wise convergence \eqref{hm:gqmle-t3}, so that the stochastic Ascoli-Arzel\`{a} theorem \cite[Theorem 7.3]{Bil99} applies.
This fact will be repeatedly used in the sequel without mentioning.

\medskip

Now, we turn to proving \eqref{hm:gqmle-t2}.
The consistency $\tes\cip\tz$ can be derived in a similar manner to \cite{Kes97}, through applying the argmax theorem twice.
At first stage, we define
\begin{align}
	\mbby_{1,n}(\gam) &= \frac1n\left( \mbbh_n(\res,\gam)-\mbbh_n(\res,\gam_0)\right), \nn\\
	\mbby_{1,0}(\gam) &= -\frac12 \left(\log \frac{\gam}{\gam_0}+\frac{\gam_0}{\gam}-1\right),
	\nonumber
\end{align}
the former being maximized at $\ges$. The argmax theorem ensures $\ges\cip\gam_0$ if we show that $\sup_\gam|\mbby_{1,n}(\gam) - \mbby_{1,0}(\gam)|\cip 0$ and $\argmax\,\mbby_{1,0}=\{\gam_0\}$.
The latter is trivial, and the former can be seen as follows: since $h\E_{\tz}^{j-1}\{(\chi'_{j})^2\}=\sig_{j-1}^2(\tz)$ with $\chi'_{j}$ given by \eqref{hm:def_chi1}, applying Burkholder's inequality and \eqref{hm:LLN} we obtain
\tcr{
\begin{align}
	\mbby_{1,n}(\gam) &= -\frac{1}{2n} \sumj\left\{ \log\left(\frac{\sig_{j-1}^2(\rho_0,\gam)}{\sig_{j-1}^2(\tz)}\right)
	+\left(\sig_{j-1}^{-2}(\rho_0,\gam) - \sig_{j-1}^{-2}(\tz)\right)\,(\chi'_{j})^2h\right\} \nn\\
	&= -\frac{1}{2n} \sumj\bigg\{ \log\frac{\gam}{\gam_0}
	+\bigg(\frac{1}{\gam}-\frac{1}{\gam_0}\bigg) c_{j-1}^{-1}(\rho_0) (\chi'_{j})^2h \bigg\}\nn\\
	&=-\frac{1}{2n} \sumj\bigg\{ \log\frac{\gam}{\gam_0}
	+\bigg(\frac{\gam_0}{\gam}-1\bigg)\bigg\} + O^\ast_p\left(\frac{1}{\sqrt{n}}\right) \nn\\
	&= \mbby_{1,0}(\gam) + o^\ast_p(1).
	\label{hm:gqmle-t1}
\end{align}
As for the consistency of $\res=(\aes,\bes)$, we introduce
\begin{align}
	\mbby_{2,n}(\rho) &:= \frac{1}{T_n}\left( \mbbh_n(\rho,\ges)-\mbbh_n(\rho_0,\ges)\right), \nn\\
	\mbby_{2,0}(\rho) &:= -\frac{1}{2\gam_0} (\rho-\rho_0)^\top 
	\int 
	\begin{pmatrix}
		x^{-1} & -1 \\ -1 & x
	\end{pmatrix}
	\pi_{0}(dx) (\rho-\rho_0).
	\nn
\end{align}
Obviously, we have $\argmax_{\rho}\,\mbby_{2,n}(\rho)=\{\rho_0\}$.
Some manipulation gives the following decomposition:
\begin{align}
	\mbby_{2,n}(\rho) &= -\frac{1}{2T_n}\sumj \log\left(\frac{c_{j-1}(\rho)}{c_{j-1}(\rho_0)}\right)
-\frac{\ges^{-1}}{2T_n} \sumj \Big(c_{j-1}^{-1}(\rho)h\left\{(\chi'_{j}(\rho))^2 - (\chi'_{j})^2\right\}
	\nn\\
	&{}\qquad 
	+h(c_{j-1}^{-1}(\rho)-c_{j-1}^{-1}(\rho_0))(\chi'_{j})^2\Big)
	\nn\\
	&= \mbby_{2,n}^{(1)}(\rho) + \mbby_{2,n}^{(2)}(\rho) + \mbby_{2,n}^{(3)}(\rho),
	\nonumber
\end{align}
where, letting $g_{j-1}(\rho):=c_{j-1}(\rho_0)/c_{j-1}(\rho)$,
\begin{align}
\mbby_{2,n}^{(1)}(\rho) &:= \frac12\left(-\frac{1}{T_n} \sumj \log g_{j-1}(\rho) 
+ \frac{\ges^{-1}}{T_n} \sumj (g_{j-1}(\rho)-1) h (\chi'_j)^2 \right),
\nn\\
\mbby_{2,n}^{(2)}(\rho) &:=
\frac{\ges^{-1}}{\sqrt{T_n}} \, \left\{\frac{1}{\sqrt{n}}\sumj \pi''_{j-1}(\rho) \left(\frac{\mu_{j-1}(\rho)-\mu_{j-1}}{h}\right)
	\chi'_j \right\},
\nn\\
\mbby_{2,n}^{(3)}(\rho) &:= - \frac{\ges^{-1}}{2n} \sumj \pi''_{j-1}(\rho) \left(\frac{\mu_{j-1}(\rho)-\mu_{j-1}}{h}\right)^2.
\nonumber
\end{align}
In an analogous manner to \eqref{hm:gqmle-t1} and through the tightness criterion as in the proof of Lemma \ref{hm:lem_cLSE}, we can obtain $\mbby_{2,n}^{(2)}(\rho)= O_p^\ast(T_n^{-1/2})=o_p^\ast(1)$ and $\mbby_{2,n}^{(3)}(\rho) =\mbby_{2,0}(\rho) + o_p^\ast(1)$.
To see that $\mbby_{2,n}^{(1)}(\rho)$, we note the inequality $0\le -(\log x) + x-1\le(2^{-1} \vee x^{-1})(x-1)^2$ valid for $x>0$
and the bound $|g_{j-1}(\rho)-1|\lesssim h(1+X_{t_{j-1}}^{-1})$.
Using them together with the consistency of $\ges$ and the tightness argument as before, we obtain
\begin{align}
\mbby_{2,n}^{(1)}(\rho) &= 
\frac{1}{2 T_n} \sumj \Big(-\log g_{j-1}(\rho) + (g_{j-1}(\rho)-1)\Big) \nn\\
&{}\qquad - \frac{\ges^{-1}(\ges-\gam_0)}{2T_n} \sumj (g_{j-1}(\rho)-1)+O^\ast_p\left(\frac{1}{\sqrt{n}}\right)
\nn\\
&=O^\ast_p\left(\frac{nh^2}{T_n}\right)+o^\ast_p(1) +O^\ast_p\left(\frac{1}{\sqrt{n}}\right) = o^\ast_p(1).
\nonumber
\end{align}
}
After all we have derived $\mbby_{2,n}^{(1)}(\rho)=\mbby_{2,0}(\rho)+o_p^\ast(1)$, hence followed by $\res\cip\rho_0$.

\medskip

Having the consistency of $\tes$ in hand, we proceed with the standard route through the third-order Taylor expansion of $\p_\theta\mbbh_n(\tes)$ around $\tz\in\Theta$. That is, we may and do focus on the event $\{\p_\theta\mbbh_n(\tes)=0\}$, on which
\begin{equation}
	0 = D_n^{-1}\p_\theta\mbbh_n(\tz) + \left(
	-\int_0^1 D_n^{-1}\p_\theta^2\mbbh_n\left(\tz+s(\tes-\tz)\right)D_n^{-1} ds \right) [D_n(\tes-\tz)].
	\nonumber
\end{equation}
Hence, by the consistency $\tes\cip\tz$, it suffices to verify the following statements:
\begin{description}
	\item[{\rm (AN1)}] $\ds{\D_n(\tz) := D_n^{-1}\p_\theta\mbbh_n(\tz) \cil N_p\left(0,\mci(\tz)\right)}$;
	\item[{\rm (AN2)}] $\ds{\mci_n(\tz) := -D_n^{-1}\p_\theta^2\mbbh_n(\tz) D_n^{-1} \cip \mci(\tz)}$;
	\item[{\rm (AN3)}] $\ds{\sup_{\theta:\,|\theta-\tz|\le\del_n} \left|\mci_n(\theta)-\mci(\tz)\right| \cip 0}$ for any (non-random) positive sequence $\del_n\to 0$.
\end{description}
Let us look at the the partial derivatives $\p_\theta^k\mbbh_n(\theta)=\p_{(\rho,\gam)}^k\mbbh_n(\rho,\gam)$ for $k=1,2$ through the expression \eqref{hm:def_GQLF}: with the notation introduced in Section \ref{hm:sec_preliminary.est},
\begin{align}
	\mbbh_n(\theta) &= C_n -\frac12 \sumj \left(
	\log\gam + \log c_{j-1}(\rho) + \frac{h}{\gam c_{j-1}(\rho)}
	\chi'_{j-1}(\rho)^2 \right).
	\nonumber
\end{align}
The calculations will be elementary, yet tedious.
For notational convenience we will write $A_k(x;\rho,h)$, $k\ge 0$, for a generic matrix-valued measurable function on $\mbbr\times\Theta_\rho\times (0,1]$ such that
\begin{equation}
	\exists C>0,\quad \limsup_{h\to 0}\,\sup_\rho|A_k(x;\rho,h)| \lesssim \frac{1+|x|^C}{1\wedge |x|^k}.
	\label{hm:A_property}
\end{equation}
Further, we write $\mca_k(\rho,h)$ for the set of all $A_k(x;\rho,h)$ satisfying \eqref{hm:A_property}.
It should be noted that by the assumption \eqref{hm:param.assumption} and Lemma \ref{hm:lem_AKT}, for each $k<5$ we can find a (small) constant $\ep_k>0$ such that
\begin{equation}
	\sup_t \E\left( |A_k(X_t;\rho,h)|^{1+\ep_k} \right)<\infty.
	\label{hm:A_property2}
\end{equation}


\begin{proof}[Proof of (AN1)]
	Let $\zeta_j:=(\chi'_j,\chi''_j)=(\chi'_j(\rho_0),\chi''_j(\tz))$.
	In what follows, we abbreviate  $\p_\rho$ as  ``overhead dot'' (like $\p_\rho \mu_{j-1}(\rho)$ as $\dot{\mu}_{j-1}(\rho)$), respectively, and moreover, omit ``$(\tz)$'' and ``$(\rho_0)$'' from the notation.
	Then, straightforward calculations lead to
	\begin{align}
		\D_n(\tz) &= 
		D_n^{-1}\sumj
		\begin{pmatrix}
			\frac{\dot{\mu}_{j-1} \sqrt{h}}{\gam_0 c_{j-1}} & \frac{\dot{c}_{j-1} h}{2\gam_0 c_{j-1}^2} \\
			0 & \frac{h}{2\gam_0^2 c_{j-1}}
		\end{pmatrix}
		[\zeta_j]
		=
		\frac{1}{\sqrt{n}}\sumj
		\begin{pmatrix}
			\frac{\dot{\mu}_{j-1}}{\gam_0 c_{j-1}} & \frac{\dot{c}_{j-1} h}{2\gam_0 c_{j-1}^2} \\
			0 & \frac{h}{2\gam_0^2 c_{j-1}}
		\end{pmatrix}
		[\zeta_j].
		\nonumber
	\end{align}
	Note that $\frac{\dot{\mu}_{j-1}}{\gam_0 c_{j-1}}$ and $\frac{h}{2\gam_0^2 c_{j-1}}$ belong to the class $\mca_1(\rho,h)$ as functions of $X_{t_{j-1}}$, while $\frac{\dot{c}_{j-1}}{2\gam_0 c_{j-1}^2}$ to $\mca_1(\rho,h)$.
	By the Burkholder and H\"{o}lder inequalities combined with the $L^q$-property of $(\zeta_j)$, we obtain, for $\del>0$ small enough,
	\begin{align}
		& \E\left(\left|\frac{1}{\sqrt{n}}\sumj \frac{\dot{c}_{j-1} }{2\gam_0 c_{j-1}^2} [\zeta_j] \right|^2 \right) \nn\\
		&\lesssim \frac1n \sumj \E\left(\left|\frac{\dot{c}_{j-1} }{2 c_{j-1}^2} [\zeta_j] \right|^2 \right)
		\nn\\
		&\lesssim \frac1n \sumj \E\left(\left|\frac{\dot{c}_{j-1}/h^2 }{2 (c_{j-1}/h)^2}\right|^{2+\del}\right)^{2/(2+\del)}
		\E\left(\left| \zeta_j\right|^{2(2+\del)/\del}\right)^{\del/(2+\del)} \lesssim 1,
		\label{AN1_p1}
	\end{align}
	where we used \eqref{hm:A_property2}.
	
	The estimate \eqref{AN1_p1} entails that the off-diagonal part in the expression of $\D_n(\tz)$ is of order $O_p(h)$, hence is asymptotically negligible.
	Write $\E^{j-1}:=\E^{j-1}_{\tz}$, and also \tcr{$x^{\otimes2} := xx^{\top}$ for any vector or matrix $x$.}
	Using Lemma \ref{hm:lem_AKT} and the same argument as in \eqref{hm:gqmle-t1}, we see that the leading term of the quadratic characteristic 
	of $\D_n(\tz)$ equals
	\begin{align}
		& \frac{1}{n}\sumj
		\begin{pmatrix}
			\frac{\dot{\mu}_{j-1}}{\gam_0 c_{j-1}} & 0 \\
			0 & \frac{h}{2\gam_0^2 c_{j-1}}
		\end{pmatrix}
		\E^{j-1}\left(\zeta_j^{\otimes 2}\right)
		\begin{pmatrix}
			\frac{\dot{\mu}_{j-1}}{\gam_0 c_{j-1}} & 0 \\
			0 & \frac{h}{2\gam_0^2 c_{j-1}}
		\end{pmatrix}^{\top} \nn\\
		&=
		\frac{1}{n}\sumj
		\begin{pmatrix}
			\frac{\dot{\mu}_{j-1}}{\gam_0 c_{j-1}} & 0 \\
			0 & \frac{h}{2\gam_0^2 c_{j-1}}
		\end{pmatrix}
		\begin{pmatrix}
			\sig_{j-1}^2(\tz)/h & (\chi'_j)^3 \sqrt{h} - (\sig_{j-1}^2/h) \sqrt{h}\, \chi'_j \\
			\mathrm{sym.} & \E^{j-1}\left\{((\chi'_j)^2 - \sig^2_{j-1}/h)^2\right\}
		\end{pmatrix}
		\nn\\
		&{}\qquad \times
		\begin{pmatrix}
			\frac{\dot{\mu}_{j-1}}{\gam_0 c_{j-1}} & 0 \\
			0 & \frac{h}{2\gam_0^2 c_{j-1}}
		\end{pmatrix}^{\top}
		\nn\\
		&= 
		\frac{1}{n}\sumj
		\begin{pmatrix}
			\frac{\dot{\mu}_{j-1}^{\otimes 2}}{h \sig^2_{j-1}} & 0 \\
			0 & \frac{1}{2\gam_0^2}
		\end{pmatrix}
		+ o_p(1) \nn\\
		&= 
		\frac{1}{n}\sumj
		\mathrm{diag}
		\left\{
		\frac{1}{\gam_0}
		\begin{pmatrix}
			X_{t_{j-1}}^{-1} & -1 \\ -1 & X_{t_{j-1}}
		\end{pmatrix},~\frac{1}{2\gam_0^2}
		\right\}
		+ o_p(1) \nn\\
		&=
		\frac{1}{n}\sumj
		\mathrm{diag}
		\left\{
		\frac{1}{\gam_0}\int_0^\infty 
		\begin{pmatrix}
			x^{-1} & -1 \\ -1 & x
		\end{pmatrix} \pi_0(dx),~\frac{1}{2\gam_0^2}
		\right\}
		+ o_p(1) \cip \mci(\tz).
		\nonumber
	\end{align}
	In a quite similar manner to \eqref{AN1_p1}, we can pick a $\del\in(0,3)$ for which the H\"older inequality guarantees
	\begin{align}
		& \sumj \E\left(\left|\frac{1}{\sqrt{n}}\frac{\dot{\mu}_{j-1}}{\gam_0 c_{j-1}}[\zeta_j]\right|^{2+\del}\right)
		+ \sumj \E\left(\left|\frac{1}{\sqrt{n}}\frac{h}{2\gam_0^2 c_{j-1}}[\zeta_j]\right|^{2+\del}\right) \nn\\
		&\lesssim \frac{1}{n^{\del/2}} \cdot \frac1n\sumj 1 \lesssim \frac{1}{n^{\del/2}} \to 0.
		\nonumber
	\end{align}
	This verifies the Lyapunov condition, and the martingale central limit theorem concludes (AN1).
\end{proof}

The arguments concerning moment estimates in the above proof will be repeatedly used in the proofs of (AN2) and (AN3), hence we will proceed without mentioning the full details.

\begin{proof}[Proof of (AN2)]
	By computing $\p_\theta^2\mbbh_n(\theta)$, we observe that
	\begin{align}
		-\frac{1}{T_n}\p^2_\rho\mbbh_n 
		&= \frac1n\sumj \frac{(\dot{\mu}_{j-1}/h)^{\otimes 2}}{\sig^2_{j-1}/h} +\frac1n\sumj\Big( h A_2(X_{t_{j-1}};\rho,h) \nn\\
		&{}\qquad  + \sqrt{h} A_2(X_{t_{j-1}};\rho,h)\chi'_j + h A_3(X_{t_{j-1}};\rho,h)\chi''_j
		\Big)
		\nn\\
		&=\frac1n\sumj \frac{(\dot{\mu}_{j-1}/h)^{\otimes 2}}{\sig^2_{j-1}/h} + o_p(1)
		\nn\\
		&=\frac{1}{\gam_0}\int_0^\infty 
		\begin{pmatrix}
			x^{-1} & -1 \\ -1 & x
		\end{pmatrix} \pi_0(dx) + o_p(1),
		\nn\\
		-\frac{1}{n\sqrt{h}}\p_\rho\p_\gam\mbbh_n &= 
		\frac1n\sumj \Big( \sqrt{h} A_1(X_{t_{j-1}};\rho,h) + A_1(X_{t_{j-1}};\rho,h)\chi'_j \nn\\
		&{}\qquad + \sqrt{h} A_2(X_{t_{j-1}};\rho,h)\chi''_j \Big) \nn\\
		&=o_p(1), \nn\\
		-\frac{1}{n}\p^2_\gam\mbbh_n &= \frac{1}{2\gam_0^2} + \frac1n\sumj A_2(X_{t_{j-1}};\rho,h) \chi''_j
		= \frac{1}{2\gam_0^2} + o_p(1).
		\nonumber
	\end{align}
	Applying Lemma \ref{hm:lem_AKT} with \eqref{hm:A_property2} to these expressions yields (AN2): $\mci_n(\tz) \cip \mci(\tz)$.
\end{proof}

\begin{proof}[Proof of (AN3)]
	To show the claim, we simply look at the uniform-in-$\theta$ behavior of $\p_\theta^3\mbbh_n(\theta)$.
	Further by continuing the partial differentiations from $\p_\theta^2\mbbh_n(\theta)$ and again applying Lemma \ref{hm:lem_AKT} with \eqref{hm:A_property2}, it is not difficult to deduce the order estimate $\sup_\theta \left|\p_\theta\mci_n(\theta)\right| = O_p(1)$.
	The claim follows on noting that
\begin{align}
\sup_{\theta:\,|\theta-\tz|\le\del_n} \left|\mci_n(\theta) -\mci(\tz)\right|
&\lesssim \left|\mci_n(\tz) -\mci(\tz)\right| + \del_n \sup_\theta \left|\p_\theta\mci_n(\theta)\right| \nn\\
&=o_p(1) + O_p(\del_n) = o_p(1)
\nonumber
\end{align}
by (AN2).
\end{proof}

\subsection{One-step improvement}

We have seen that $D_n(\tesz-\tz) = O_p(1)$ and $D_n(\tes-\tz) \cil N_p\left(0,\, \mci(\tz)^{-1}\right)$ in Lemmas \ref{hm:lem_cLSE} and \ref{hm:lem_AN.GQMLE}, respectively.
Let us now look at the Newton-Raphson method and the method of scoring, from $\tesz$ to the joint GQMLE $\tes$ defined by \eqref{hm:def_GQMLE}:
\begin{align}
	\hat{\theta}_{n}^{(1,1)} &= \tesz-\left(\p_{\theta}^2\mbbh_n(\tesz)\right)^{-1}\p_{\theta}\mbbh_n(\tesz),
	\label{onestep-NR}\\
	\hat{\theta}_{n}^{(1,2)} &= \tesz+D_{n}^{-1}\mathcal{I}(\tesz)^{-1}D_{n}^{-1}\p_{\theta}\mbbh_n(\tesz).
	\label{onestep-scoring}
\end{align}
Observe that the following statements hold from what we have obtained so far, in particular (AN1) to (AN3):
\begin{align}
	& |\mci_n(\tesz)^{-1}| = O_p(1), \nn\\
	& |I_p - \mci_n(\tesz)^{-1}\mci_n(\tesz')| = o_p(1),
	\nonumber
\end{align}
where $\tesz'$ is any random point on the segment connecting $\tesz$ and $\tes$.
Let $\desz := D_n(\tesz-\tes)$; we have $\desz=O_p(1)$. \tcr{Recall that we are focusing on the event $\{\p_\theta\mbbh_n(\tes)=0\}$, so that $\Delta_n(\tes)=0$.} For some $\tesz'$ as above,
\begin{align}
	D_n(\hat{\theta}_{n}^{(1,1)} - \tes) 
	&= \desz + \mci_n(\tesz)^{-1} \D_n(\tesz) \nn\\
	&= \desz + \mci_n(\tesz)^{-1} \left\{\D_n(\tes) + D_n^{-1}\p_\theta^2\mbbh_n(\tesz')[\tesz-\tes]\right\} \nn\\
	&= \left( I_p - \mci_n(\tesz)^{-1}\mci_n(\tesz') \right) \desz  = o_p(1).
	\nonumber
\end{align}
Hence, we obtained $D_n(\hat{\theta}_{n}^{(1,1)}-\tes) = o_p(1)$.
An analogous argument leads to $D_n(\hat{\theta}_{n}^{(1,2)}-\tes) = o_p(1)$, and we conclude the following theorem.

\begin{theorem}
	\label{thm_GQMLEs}
	We have $D_n(\hat{\theta}_{n}^{(1,i)}-\tes) = o_p(1)$ for $i=1,2$, hence in particular,
	the asymptotic normality
	\begin{align}
		D_n(\hat{\theta}_{n}^{(1,i)}-\tz) \cil N_3\left(0,\,\mci(\tz)^{-1}\right)
		\nonumber
	\end{align}
	and their Studentized versions
	\begin{equation}
		\mathcal{I}(\hat{\theta}_{n}^{(1,i)})^{1/2}D_n(\hat{\theta}_{n}^{(1,i)}-\tz) \cil N_3(0, I_3)
		\nonumber
	\end{equation}
	hold.
\end{theorem}

We end this section with several remarks.

\begin{remark}
	\label{rem_GQMLEs-1}
	\begin{enumerate}
		\item The previous study \cite{AlaKebTra20} derived the local asymptotic normality for $(\al,\beta)$, assuming that the diffusion parameter $\gam$ is known:
		\cite[Theorem 1]{AlaKebTra20} says that, if $\tz\in\Theta$ further satisfies that
		\begin{equation}
			\frac{2\al_0}{\gam_0} > \frac{11+\sqrt{89}}{2} \approx 10.217,
			\nonumber
		\end{equation}
		then any regular estimator $(\aes^\star,\bes^\star)$ such that
		\begin{equation}
			\left(\sqrt{T_n}(\aes^\star -\al_0),\,\sqrt{T_n}(\bes^\star-\beta_0)\right)
			\cil N_2\left(0,\,
			\left\{
			\frac{1}{\gam_0}
			\begin{pmatrix}
				\frac{\beta_0}{\al_0-\gam_0/2} & -1 \\ -1 & \frac{\al_0}{\beta_0}
			\end{pmatrix}
			\right\}^{-1}
			\right)
			\nonumber
		\end{equation}
		is asymptotically efficient in the sense of Haj\'{e}k-Le Cam.
Theorem \ref{thm_GQMLEs} suggests that the asymptotic optimality of the GQMLE would remain valid under a weaker restriction on the parameter space; recall \eqref{hm:param.assumption}.
On the one hand, the optimality of the GMQLE may seem unnatural since the small-time Gaussian approximation is quite inappropriate for the non-central chi-square distribution. On the other hand, it is natural since the driving noise is Gaussian, so that the model would be approximately Gaussian in small time. Note that this observation is just an intuition, and does not theoretically run counter to anything.
		
		\item It is well-known in the literature that high-frequency data over any bounded domain, say $[0,T]$ for fixed $T>0$ (namely $T_n\equiv T$), is enough to consistently estimate $\gam$ even when the diffusion coefficient is not uniformly elliptic: see the seminal paper \cite{genon1993} for details. 
		We may set $t_j=Tj/n$ for a fixed $T>0$, and then the modified quasi-log likelihood
		\begin{equation}
			\gam \mapsto \sumj \log \phi\left( X_{t_j};\, 0,\, h\gam X_{t_{j-1}}
			\right),
			\nonumber
		\end{equation}
		where the drift coefficient is completely ignored, would provide us with an asymptotically normally distributed estimator of $\gam$ with convergence rate being $\sqrt{n}$.
		In this case, the drift coefficient may be regarded as an infinite-dimensional nuisance element, while we cannot consistently estimate it in theory.
		
		\item Since we know beforehand that $(\aes,\bes)$ and $\ges$ are asymptotically independent, we may replace $\p_{\theta}^2\mbbh_n(\tesz)$ in \eqref{onestep-NR} by the simpler block-diagonal form
		\begin{equation}
			\mathrm{diag}\left(\p_{\rho}^2\mbbh_n(\tesz),\,\p_{\gam}^2\mbbh_n(\tesz)\right).
			\nonumber
		\end{equation}
		
	\end{enumerate}
\end{remark}

\section{Numerical experiments}
\label{nh:sec_simulations}

In this section, we compare simulation results of the initial estimator \eqref{hm:aesz_def}, \eqref{hm:besz_def}, \eqref{hm:gesz_def} with the one step estimators based on the Newton-Raphson \eqref{onestep-NR} and the scoring method \eqref{onestep-scoring}. We use exact CIR simulator for $(X_{t_j})_{j=1,\dots,n}$ through non-central chi-squares \cite{MalWie13}. The 1.000 simulated estimators are performed for $n=5.000, 10.000, 20.000$ and $T=T_n=500, 1.000, 2.000$. So that the condition $\frac{2\al}{\gam}>5$ is fulfilled, we choose as true value $\theta_0=(\al_0,\beta_0,\gam_0)=(3,1,1)$. Table \ref{table1} summarizes the mean and standard deviation (sd) of these estimators. As a comparison for the behavior at smaller $T$ and $n$, we contrast boxplots for $T=500, n=5.000$ and $T=50, n=200$ (Figures \ref{fig1} to \ref{fig3}). Additionally, the corresponding histograms are given in Figure \ref{fig4} and \ref{fig5}.
 
\begin{table}
	\caption{The mean and the standard deviation (sd) of the estimators with true values $(\al_0,\beta_0,\gam_0)=(3,1,1)$.} 
	\label{table1}
	{\tiny 
	\begin{align*}
		\begin{array}{cccccccccc}
			\hline
			n=5.000& & T=500 & & &T=1.000& & &T=2.000 &\\
			\textrm{initial}& \hat{\al}_{0,n} &\hat{\beta}_{0,n} &\hat{\gam}_{0,n}& \hat{\al}_{0,n} &\hat{\beta}_{0,n} &\hat{\gam}_{0,n}& \hat{\al}_{0,n} &\hat{\beta}_{0,n} &\hat{\gam}_{0,n}\\
			\hline
			\textrm{mean}& 3.0188 & 1.0079 & 0.9998 & 3.0200 & 1.0072 & 1.0023& 3.0037 &1.0015 & 1.0006\\
			\textrm{sd}&0.2111 & 0.0752 & 0.0211 & 0.1570 & 0.0552 & 0.0235 & 0.1216 & 0.0421 & 0.0254\\
			\hline	\hline
			\textrm{Newton} & \hat{\al}_{1,n} &\hat{\beta}_{1,n} &\hat{\gam}_{1,n}& \hat{\al}_{1,n} &\hat{\beta}_{1,n} &\hat{\gam}_{1,n}& \hat{\al}_{1,n} &\hat{\beta}_{1,n} &\hat{\gam}_{1,n}\\
			\hline 
			\textrm{mean}&3.0141 &1.0064 &1.0001 & 3.0197 & 1.0071 & 1.0025 & 3.0027 & 1.0012 & 1.0010\\
			\textrm{sd} & 0.1872 & 0.0681 & 0.0220 & 0.1435 & 0.0515 & 0.0248 & 0.1156 & 0.0403 & 0.0270\\
			\hline \hline
			\textrm{scoring} & \hat{\al}_{2,n} &\hat{\beta}_{2,n} &\hat{\gam}_{2,n}& \hat{\al}_{2,n} &\hat{\beta}_{2,n} &\hat{\gam}_{2,n}& \hat{\al}_{2,n} &\hat{\beta}_{2,n} &\hat{\gam}_{2,n}\\
			\hline
			\textrm{mean} & 3.0105 & 1.0052 & 0.9998 & 3.0182 & 1.0066 & 1.0023 & 3.0017 & 1.0008 & 1.0006\\
			\textrm{sd }& 0.1877 & 0.0682 & 0.0211 & 0.1434 & 0.0514 & 0.0235 & 0.1152 & 0.0400 & 0.0255 \\
			\hline \hline
			n=10.000 & &&&&&&&&\\ 
			\textrm{initial}& \hat{\al}_{0,n} &\hat{\beta}_{0,n} &\hat{\gam}_{0,n}& \hat{\al}_{0,n} &\hat{\beta}_{0,n} &\hat{\gam}_{0,n}& \hat{\al}_{0,n} &\hat{\beta}_{0,n} &\hat{\gam}_{0,n}\\ 
			\hline
			\textrm{mean} & 3.0351 & 1.0129 & 1.0001 & 3.0152 & 1.0060 & 1.0003 & 3.0097 & 1.0026 & 1.0012 \\
			\textrm{sd} & 0.2143 & 0.0741 & 0.0145 & 0.1519 & 0.0541 & 0.0150 & 0.1135 & 0.0404 & 0.0167 \\
			\hline \hline
			\textrm{Newton} & \hat{\al}_{1,n} &\hat{\beta}_{1,n} &\hat{\gam}_{1,n}& \hat{\al}_{1,n} &\hat{\beta}_{1,n} &\hat{\gam}_{1,n}& \hat{\al}_{1,n} &\hat{\beta}_{1,n} &\hat{\gam}_{1,n}\\
			\hline 
			\textrm{mean} & 3.0257 & 1.0099 & 1.0004 & 3.0121 & 1.0050 & 1.0005 & 3.0097 & 1.0026 & 1.0013 \\
			\textrm{sd} & 0.1849 & 0.0654 & 0.0150 & 0.1319 & 0.0476 & 0.0156 & 0.1030 & 0.0372 & 0.0175 \\
			\hline \hline 
			\textrm{scoring} & \hat{\al}_{2,n} &\hat{\beta}_{2,n} &\hat{\gam}_{2,n}& \hat{\al}_{2,n} &\hat{\beta}_{2,n} &\hat{\gam}_{2,n}& \hat{\al}_{2,n} &\hat{\beta}_{2,n} &\hat{\gam}_{2,n}\\ 
			\hline
			\textrm{mean} & 3.0219 & 1.0085 & 1.0001 & 3.0101 & 1.0043 & 1.0003 & 3.0089 & 1.0023 & 1.0012 \\
			\textrm{sd} & 0.1849 & 0.0654 & 0.0145 & 0.1320 & 0.0475 & 0.0150 & 0.1028 & 0.0371 & 0.0167\\
			\hline \hline	
			n=20.000& & & & && & &&\\
			\textrm{initial}& \hat{\al}_{0,n} &\hat{\beta}_{0,n} &\hat{\gam}_{0,n}& \hat{\al}_{0,n} &\hat{\beta}_{0,n} &\hat{\gam}_{0,n}& \hat{\al}_{0,n} &\hat{\beta}_{0,n} &\hat{\gam}_{0,n}\\
			\hline 
			\textrm{mean} & 3.0394 & 1.0140 & 1.0008 & 3.0102 & 1.0039 & 0.9996 & 3.0060 & 1.0023 & 1.0001\\
			\textrm{sd} & 0.2010 & 0.0717 & 0.0105 & 0.1464 & 0.0518 & 0.0104 & 0.1071 & 0.0374 & 0.0109 \\
			\hline \hline	
			\textrm{Newton} & \hat{\al}_{1,n} &\hat{\beta}_{1,n} &\hat{\gam}_{1,n}& \hat{\al}_{1,n} &\hat{\beta}_{1,n} &\hat{\gam}_{1,n}& \hat{\al}_{1,n} &\hat{\beta}_{1,n} &\hat{\gam}_{1,n}\\
			\hline 
			\textrm{mean} & 3.0263 & 1.0097 & 1.0010 & 3.0120 & 1.0045 & 0.9996 & 3.0043 & 1.0018 & 1.0002 \\
			\textrm{sd} & 0.1756 & 0.0635 & 0.0106 & 0.1269 & 0.0460 & 0.0106 & 0.0946 & 0.0334 & 0.0112\\
			\hline	\hline
			\textrm{scoring} & \hat{\al}_{2,n} &\hat{\beta}_{2,n} &\hat{\gam}_{2,n}& \hat{\al}_{2,n} &\hat{\beta}_{2,n} &\hat{\gam}_{2,n}& \hat{\al}_{2,n} &\hat{\beta}_{2,n} &\hat{\gam}_{2,n}\\ 
			\hline
			\textrm{mean} & 3.0223 & 1.0083 & 1.0008 & 3.0101 & 1.0038 & 0.9996 & 3.0034 & 1.0015 & 1.0002\\
			\textrm{sd} & 0.1759 & 0.0636 & 0.0105 & 0.1270 & 0.0460 & 0.0104 & 0.0946 & 0.0333 & 0.0109\\
			\hline	
		\end{array}
	\end{align*}
	}
\end{table}


\begin{figure}[h]
	\begin{multicols}{2}
		\includegraphics[width=0.49\textwidth,trim=0pt 30pt 0pt 30pt,clip ]{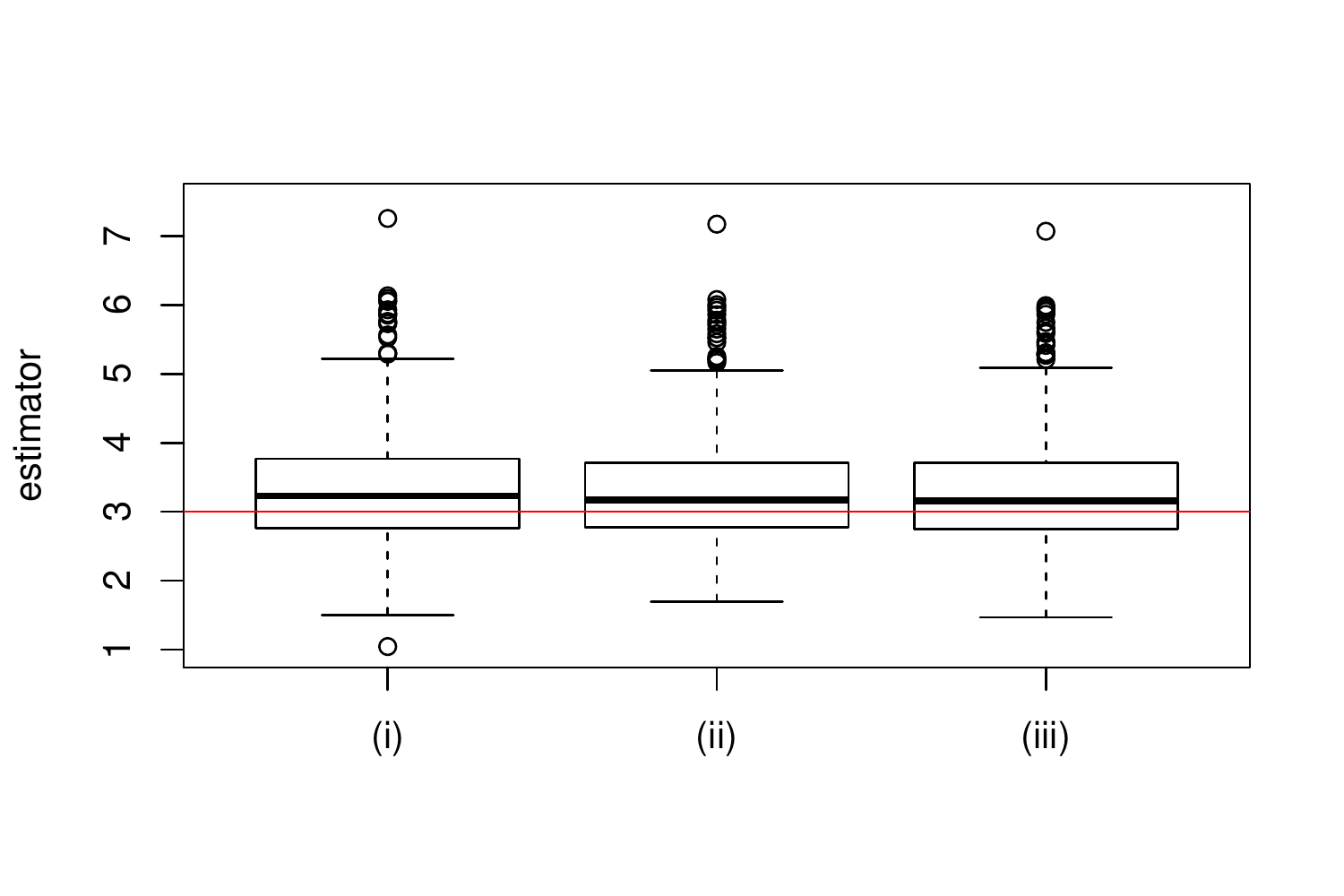}
		\newpage
		\includegraphics[width=0.49\textwidth,trim=0pt 30pt 0pt 30pt,clip]{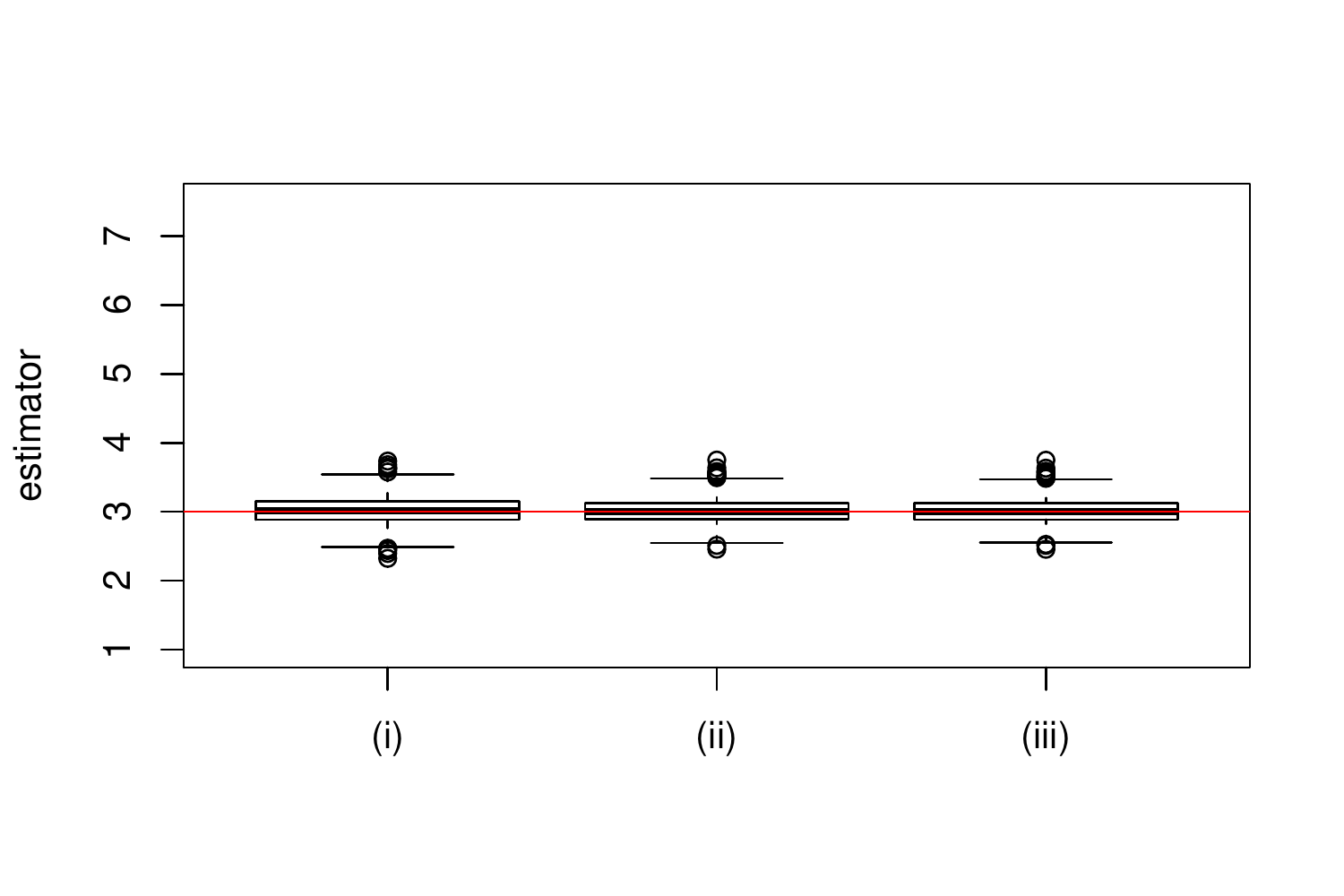}
	\end{multicols}
	\caption{Results of estimating $\alpha_0$ for $T=50, n=200$ (left) and $T=500, n=5.000$ (right): (i) initial estimator, (ii) Newton-Raphson method, (iii) scoring method. The red line indicates the true value.}
	\label{fig1}
\end{figure}

\begin{figure}[h]
	\begin{multicols}{2}
		\includegraphics[width=0.49\textwidth,trim=0pt 30pt 0pt 30pt,clip ]{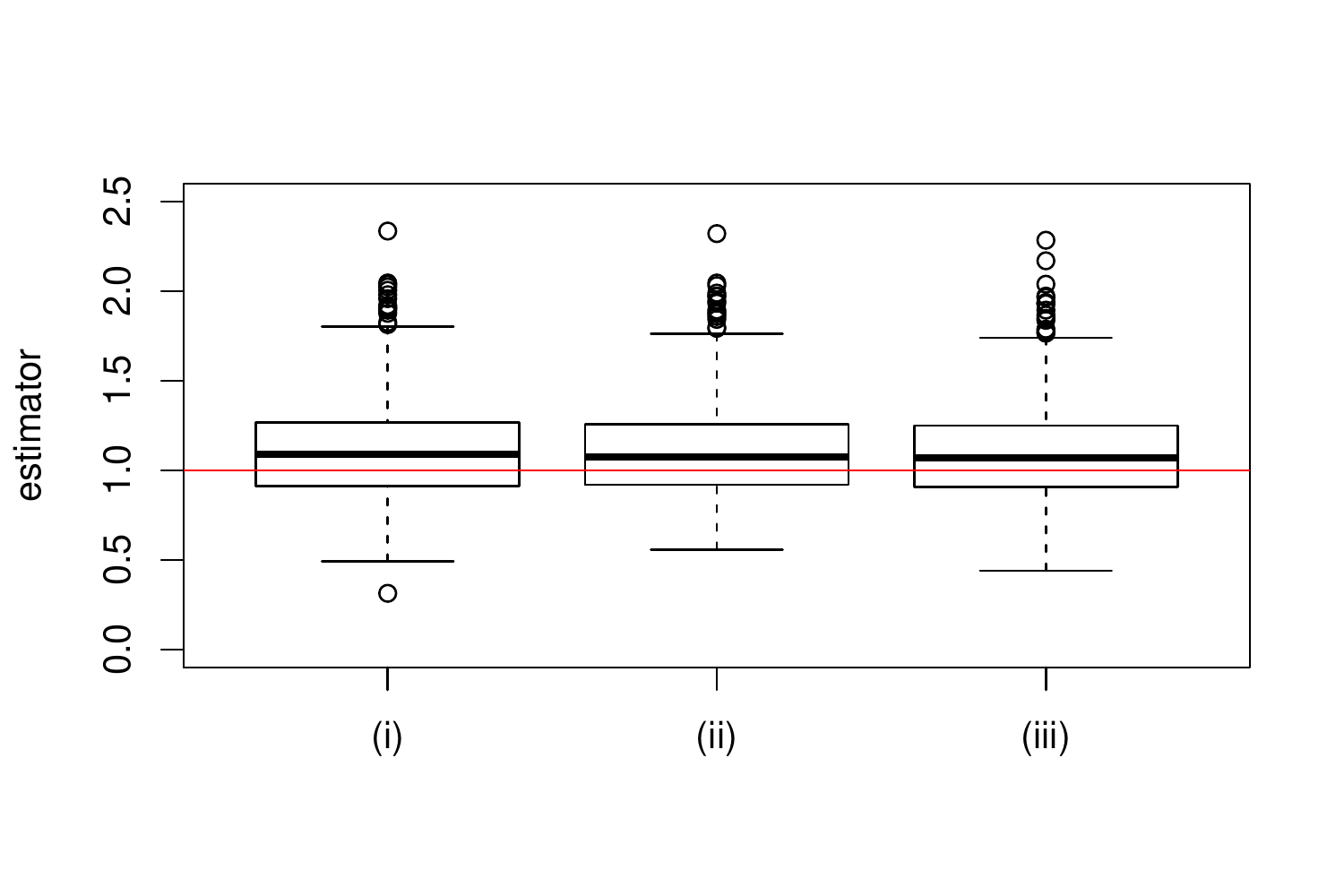}
		\newpage
		\includegraphics[width=0.49\textwidth,trim=0pt 30pt 0pt 30pt,clip]{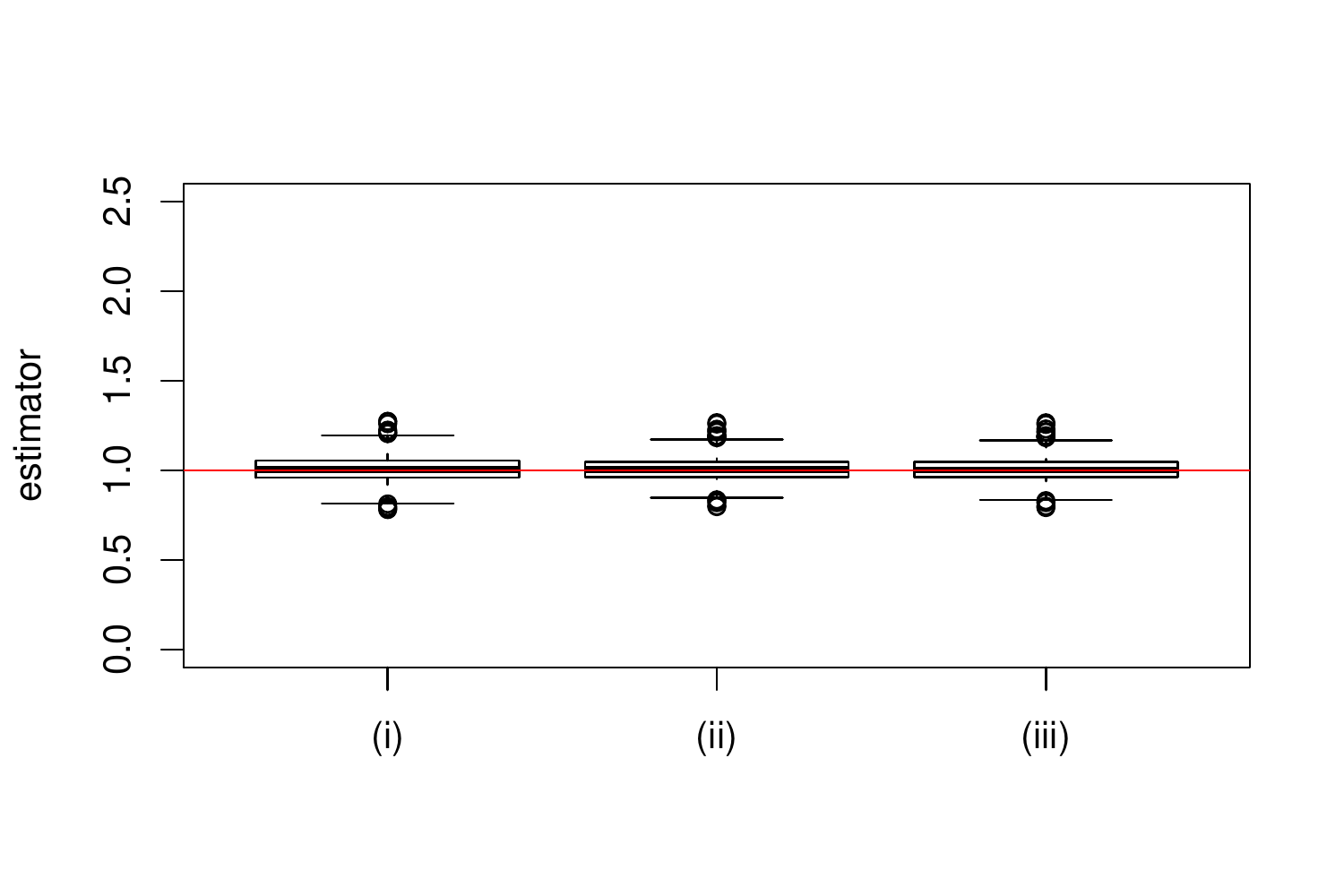}
	\end{multicols}
	\caption{Results of estimating $\beta_0$ for $T=50, n=200$ (left) and $T=500, n=5.000$ (right): (i) initial estimator, (ii) Newton-Raphson method, (iii) scoring method. The red line indicates the true value.}
	\label{fig2}
\end{figure}

\begin{figure}[h]
	\begin{multicols}{2}
		\includegraphics[width=0.49\textwidth,trim=0pt 30pt 0pt 30pt,clip ]{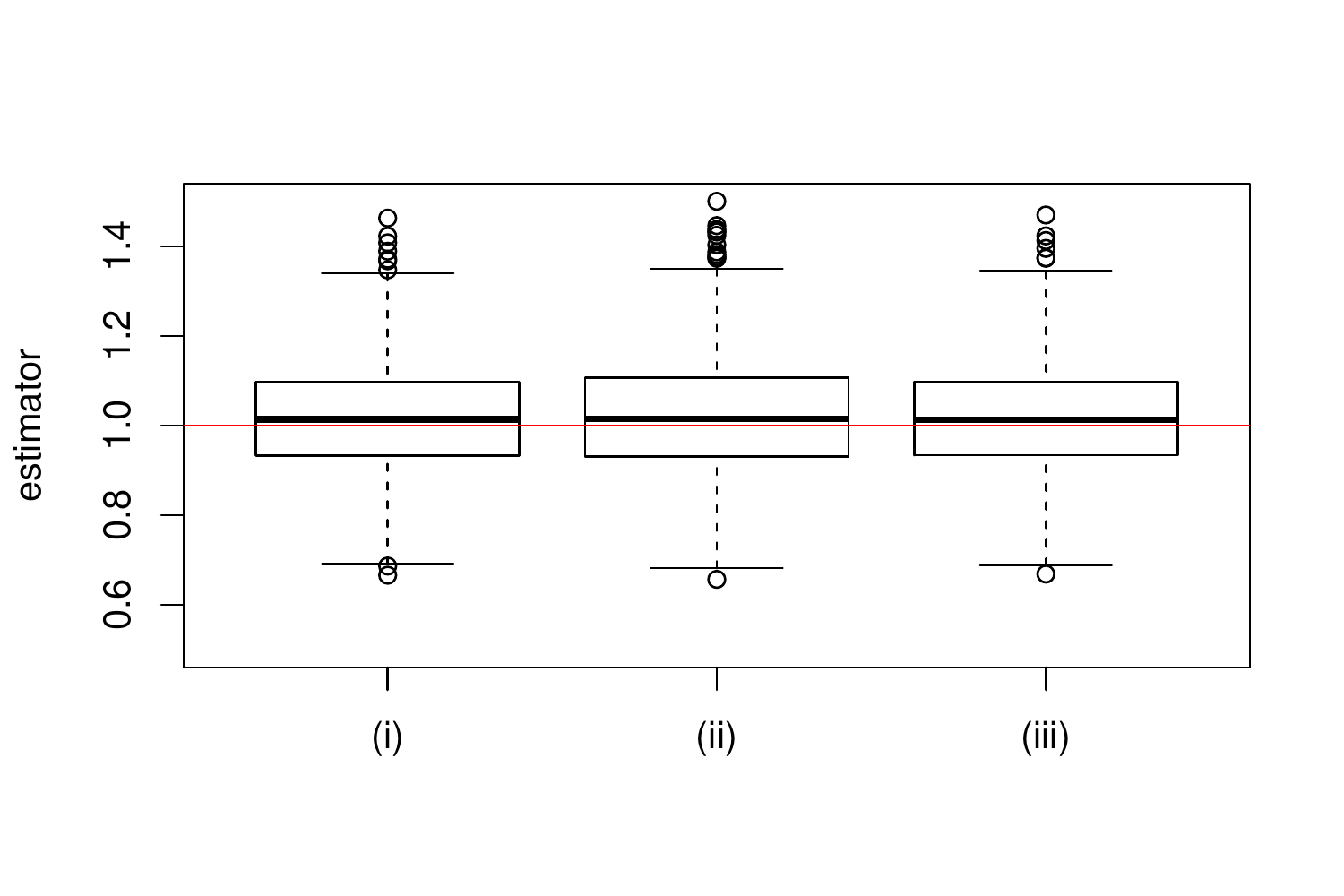}
		\newpage
		\includegraphics[width=0.49\textwidth,trim=0pt 30pt 0pt 30pt,clip]{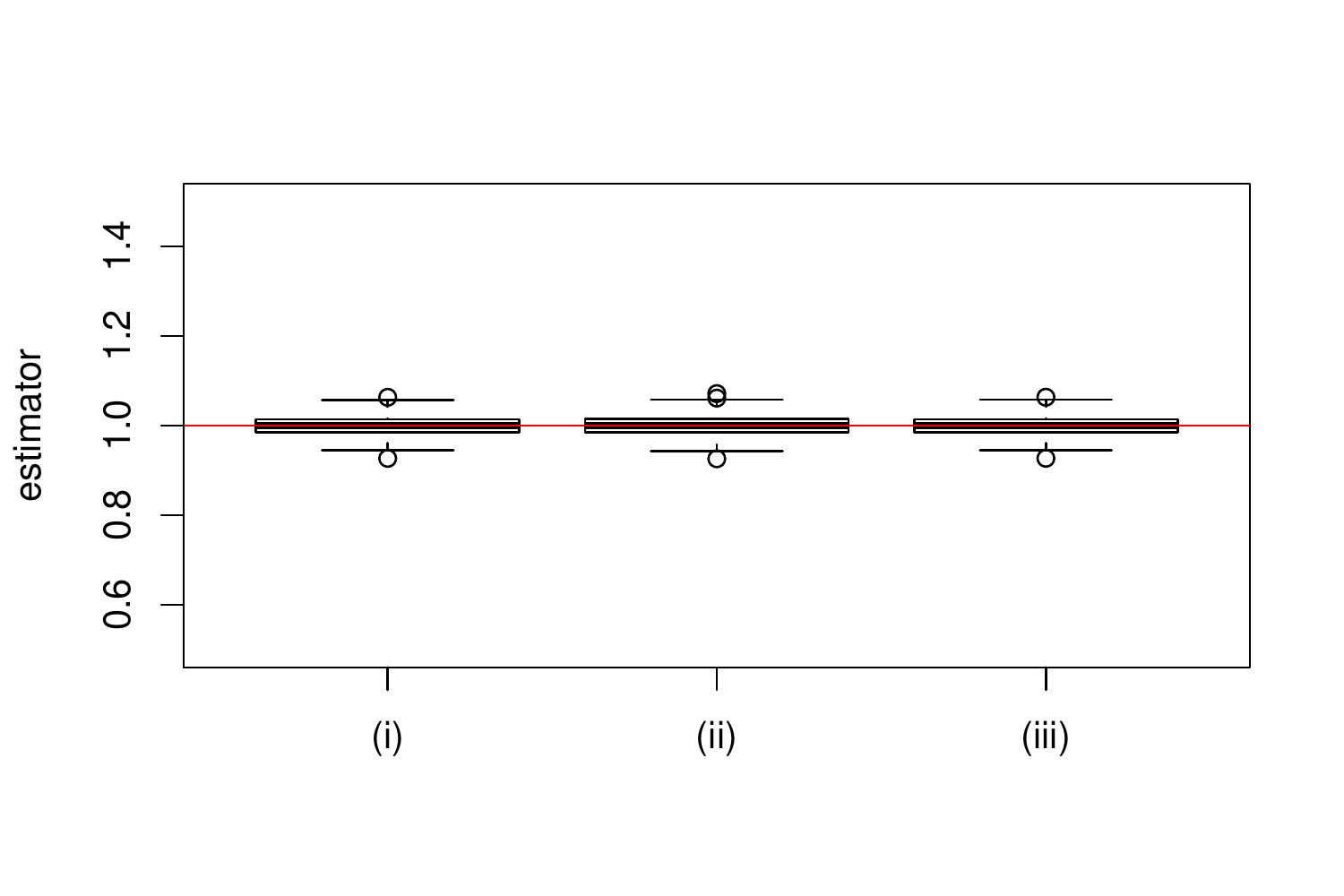}
	\end{multicols}
	\caption{Results of estimating $\gamma_0$ for $T=50, n=200$ (left) and $T=500, n=5.000$ (right): (i) initial estimator, (ii) Newton-Raphson method, (iii) scoring method. The red line indicates the true value.}
	\label{fig3}
\end{figure}


\begin{figure}[H]
	\includegraphics[width=\textwidth,trim=0pt 0pt 0pt 0 pt]{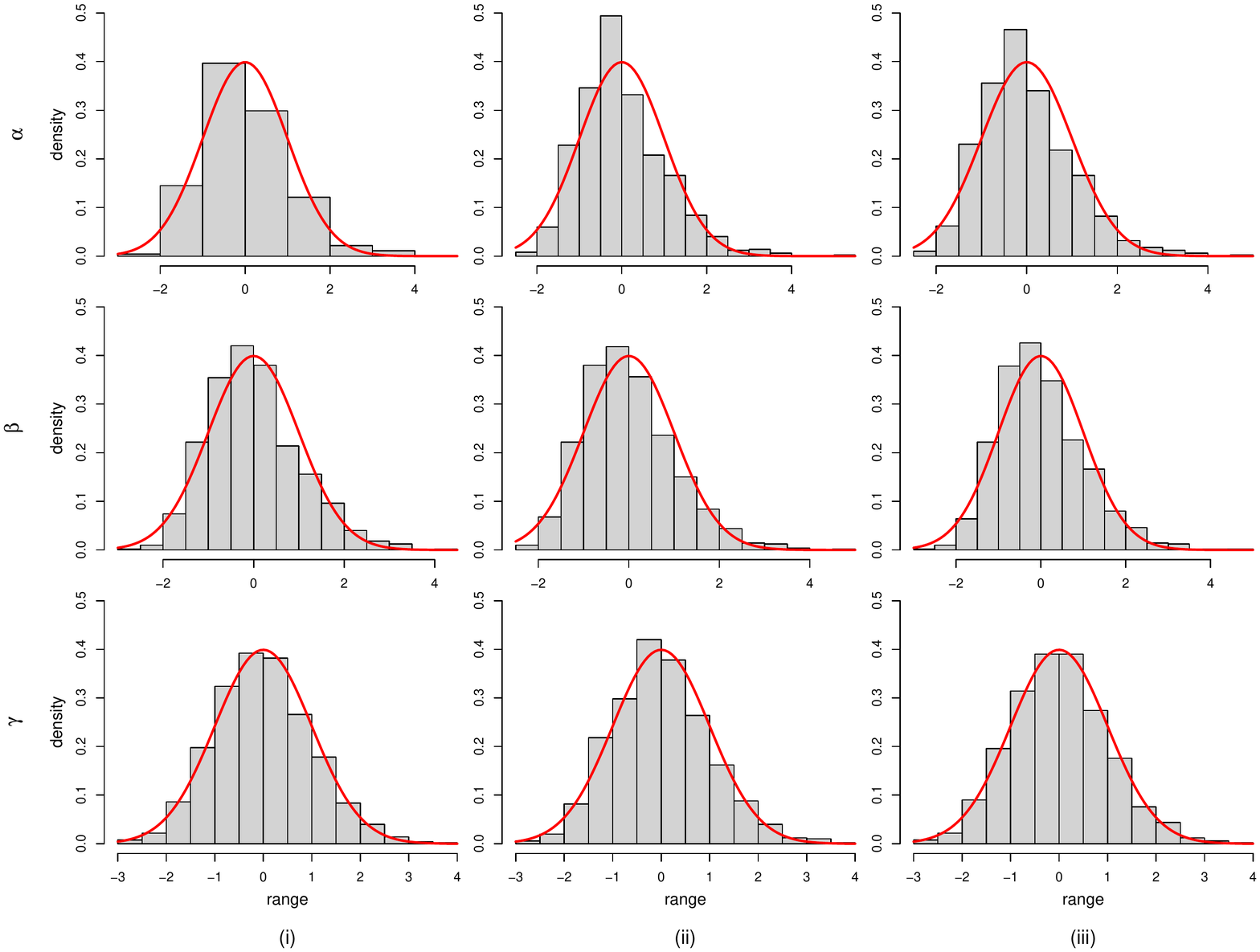}
	\caption{Histograms of the standardized estimators for $T=50, n=200$: (i) initial estimator, (ii) Newton-Raphson method, (iii) scoring method. The red curve indicates the density of the standard normal distribution. }
	\label{fig4}
\end{figure}

\begin{figure}[H]
	\includegraphics[width=\textwidth,trim=0pt 0pt 0pt 0 pt]{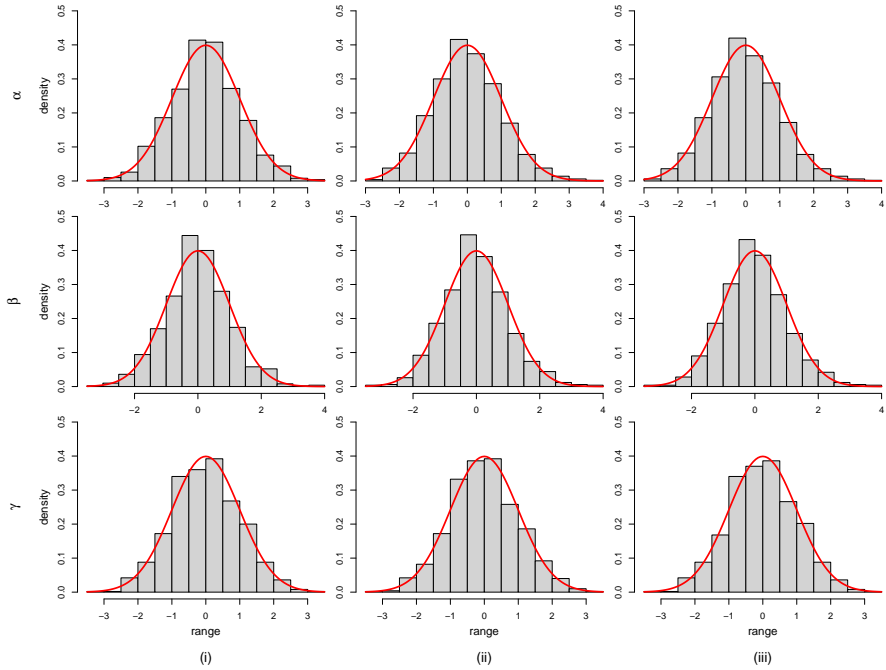}
	\caption{Histograms of the standardized estimators for $T=500, n=5.000$: (i) initial estimator, (ii) Newton-Raphson method, (iii) scoring method. The red curve indicates the density of the standard normal distribution. }
	\label{fig5}
\end{figure}

By means of Table \ref{table1} the performance of the three estimators behaves quite similar. Upon closer inspection, we recognize the smallest improvement in the estimates of $\gam_0$. The two estimators $\hat{\gam}_{0,n}$ and $\hat{\gam}_{2,n}$ have even the same values on four decimal points except for two deviations of $10^{-4}$. Comparing the estimators for $\al_0$ and $\beta_0$, we detect, with one exception (Newton-Raphson method for $T=1.000$ and $n=20.000$), a small improvement in the one step estimators compared to the initial estimator. Besides, the scoring method performs slightly better than the Newton-Raphson method. 

Overall, the three-estimators performances seem to be quite similar. This leads to the assumption that the initial estimator is already asymptotically optimal. The almost undetectable difference of the estimates of $\gam_0$ is due to the faster convergence rate $\sqrt{n}$ instead of $\sqrt{T_n}=\sqrt{T}$.

Since we choose especially large values for $n$ and $T$ in Table \ref{table1}, we can assume that the differences become more pronounced for smaller values. Therefore, comparing the boxplots for $T=50, n=200$ in Figures \ref{fig1} to \ref{fig3} the differences between the estimators are vanishingly small. This suggests that the initial estimator is already effective. The improvements in performance become visually apparent when comparing with the boxplots for $T=500, n=5.000$.

\bigskip

\noindent
{\bf Acknowledgement.} 
We are grateful to the two anonymous referees for their valuable comments which led to substantial improvements.
This work was partially supported by JST CREST Grant Number JPMJCR14D7, Japan (HM), and by DFG-GRK 2131, Germany (NH).
The contents are partly based on the master thesis of CY \cite{CY.thesis}.

\bigskip


\end{document}